%% file: main.tex
\documentclass{article}

\input{preamble.tex}

\title{Semiconjugacy and Self-Similar Subgroups of pfIMGs}
\author{Ophelia Adams}

\begin{document}
\maketitle
\begin{abstract}
    The profinite iterated monodromy group (pfIMG) is a self-similar group associated to dynamical systems. We show that its proper open self-similar subgroups correspond to highly rigid semiconjugacies, which we partly classify in general. For polynomials, we show that only the twisted Chebyshev maps can arise.  Next, we define and construct self-similar closures of subgroups of pfIMGs, and show that this preserves many group-theoretic properties of the original subgroup. As a consequence, we conclude that pfIMGs with open subgroups satisfying certain properties (e.g. prosolvable or pronilpotent) either satisfy that property themselves, or arise from one of these exceptional semiconjugacies.

    This is applied to answer some questions posed in~\cite{BGJT:specializations} about open Frattini subgroups of pfIMGs: unicritical polynomials of composite degree do not have an open Frattini subgroup, and a polynomial with an open Frattini subgroup is often pro-\(p\).
\end{abstract}

\section{Introduction}

A typical object in arithmetic dynamics is an endomorphism \(f:X\to X\) of a smooth curve defined over a field \(K\). Among the invariants of such a dynamical system are its arithmetic and geometric profinite iterated monodromy groups and its arboreal representations. The two are closely linked: arboreal representations can be realized as decomposition subgroups of the arithmetic profinite iterated monodromy group. In the present work, we study the interaction between the geometry of the original dynamical system and the self-similar group-theoretic structure of the pfIMG.

A version of Odoni's conjecture asks whether Hilbert's Irreducibility Theorem can be extended to the infinite extensions obtained by iterating a single polynomial or rational map, and whether hilbertianity is sufficient: if \(K\) is a hilbertian field, can you find infinitely many \(a\in K\) such that \(K(f^{-\infty}(a))/K\) and \(K(f^{-\infty}(t)/K(t)\) are isomorphic?

In~\cite{BGJT:specializations}, this was answered affirmatively for polynomials over a number field whose geometric profinite iterated monodromy group is pro-\(p\) by reducing the infinite specialization problem to a finite one, at which point the Hilbert Irreducibility Theorem applies.

\subsection{Summary of Results}
In Section~\ref{section: semiconjugacy}, we prove (Theorem~\ref{theorem: self-similar functoriality}) that self-similar subgroups of profinite iterated monodromy groups arise in an essentially functorial fashion from a particular kind of semiconjugacy, an induced quotient of a pullback (see Definition~\ref{definition: dynamical pullback}). The archetypal examples of such quotients are Latt\`es and twisted Chebyshev maps. Applying Riemann-Hurwitz leads to the main result of this section:
\begin{rerefresult}{Theorem~\ref{theorem: classification of open self similar pfIMG quotients}}
    Let \(\bArb\) be the geometric profinite iterated monodromy group of \(f\). Suppose \(H\) is an open self-similar subgroup of \(\Arb\) and let \(N\) be the normal core of \(H\). Let \(L\) be the fixed field of \(N\) and \(Y_N\) the corresponding curve with cover \(\pi_N:Y_N\to \PP^1_{\bar K}\). Then \(f\) is an induced pullback quotient, where \(Y=Y_N\) and \(\pi = \pi_N\).

    In particular, \(\bArb/N\) is either a finite subgroup of automorphisms of an elliptic curve and \(f\) is Latt\`es, or \(\bArb/N\) is one of the finite subgroups of \(\PGL_2\): cyclic, dihedral, \(A_4\), \(S_4\), or \(A_5\). In the latter case, \(\pi_N\) is ramified over two or three points, is determined by those ramification indices, and all three branch points must also be pre-periodic and branch points for \(f\).

    If \(f\) is a polynomial and \(H\neq \bArb\) then \([\bArb:N] = 2\), \(H=N\), \(f\) is a twisted Chebyshev map, \(\bArb\) is the infinite pro-\(d\) dihedral group, and \(H\) is the rotation subgroup, generated by the odometer.
\end{rerefresult}

Motivated by this, we define in Section~\ref{section: self similar properties} the self-similar closure and self-similar properties of subgroups of general self-similar groups (Definitions~\ref{definition: self-similar closure}~and~\ref{definition: self-similar property}). For fractal groups, which includes all profinite iterated monodromy groups, we establish some sufficient criteria for properties to be self-similar (Proposition~\ref{proposition: construction criteria for self similarity}, Proposition~\ref{proposition: formation implies ssc}) and apply them prove that various properties of open normal subgroups are self-similar: pronilpotence in Proposition~\ref{proposition: pronilpotent normal is ssc} and prosolvability and being pro-\(p\) in Corollary~\ref{corollary: induced ssc examples}.

In Section~\ref{section: main}, we combine the previous sections to prove our main result:
\begin{rerefresult}{Corollary~\ref{corollary: main results}}
Suppose \(f\) a tamely ramified rational map. If \(f\)is either a rational function and not a proper quotient of a dynamical pullback, or if \(f\) is a polynomial and not twisted Chebyshev, then:
    \begin{enumerate}
        \item If \(\bArb\) is virtually torsion, then it is torsion.
        \item If \(\bArb\) is virtually pro-\(p\), then \(G\) pro-\(p\).
        \item If \(\bArb\) is virtually prosolvable, then it is prosolvable.
        \item If \(\bArb\) is virtually pronilpotent, then it is pronilpotent.
        \item If \(\bArb\) contains the standard odometer \(\omega\) and \(\llangle \omega\rrangle\) is open, then \(\bArb\) is generated by the odometer.
        \item If \(\bArb\) contains the standard odometer \(\omega\) and the center of \(\bArb\) is open, then \(G\) is abelian.
    \end{enumerate}
    
    In Case (6), \(G\) is procyclic and generated by the odometer.
\end{rerefresult}
This follows from the more general Theorem~\ref{theorem: self-similar upgrades}: self-similar closure will spread a self-similar property from an open subgroup to the entire group unless the rational map is one of the exceptional quotients in Theorem~\ref{theorem: classification of open self similar pfIMG quotients}.

Finally, we apply our main results to answer questions raised in~\cite{BGJT:specializations}. In that paper, the authors show that Odoni's conjecture can be reduced to Hilbertianity for post-critically finite rational maps with a pro-\(p\) geometric profinite iterated monodromy group, which rests largely on the fact that the Frattini subgroup will be open. They ask whether unicritical polynomials of non-prime-power degree can have an open Frattini subgroup. We answer this negatively, and also show that the geometric pfIMG of a post-critically infinite polynomial cannot be open in Corollary~\ref{corollary: main special cases}. The main point is that the Frattini subgroup is pronilpotent, so the results of the present paper limit the possibilities for it to be open without the entire pfIMG being pronilpotent. Polynomials pfIMGs contain a self-centralizing element, the odometer, but elements of coprime order in a pronilpotent group must commute. Threading this needle is difficult for pronilpotent groups, and we prove a more general sufficient criterion in terms of the critical portrait:

\begin{rerefresult}{Corollary~\ref{corollary: open frattini implies pro-p}}
    Suppose \(f\) is a polynomial over \(K\) of degree \(d\), tamely ramified, and not conjugate to a twisted Chebyshev map. Let \(\bArb\) be its profinite iterated monodromy group, and assume it contains the standard odometer. Let \(B\) denote the (strict) critical orbit.

    Suppose there is some \(b\) in \(B\) such that \(f\inv (b)\) contains no critical points and \(f\inv(b)\cap B\) has just one element. If \(G\) is virtually pronilpotent, then it is pro-\(p\).
\end{rerefresult}

For completeness, we generalize the main pro-\(p\) specialization result of~\cite{BGJT:specializations} to all rational maps with an open Frattini subgroup. In~\cite{BGJT:specializations}, the authors directly show that the arithmetic pfIMG is pro-\(p\) if the geometric pfIMG is pro-\(p\) and the ground field is a number field. We articulate the separation between the geometric and arithmetic pieces differently. First, we prove that if the geometric pfIMG has an open Frattini subgroup, then Odoni's conjecture over the constant field extension \(\wh K_f\) reduces to hilbertianity of \(\wh K_f\). If \(K\) is hilbertian and \(f\) is PCF, then \(K_f\) is a small extension (c.f.~\cite[Proposition 19.2.1]{FJ:fieldarithmetic}), which essentially appears in~\cite{BGJT:specializations}. For the special case of polynomials, we complement this with an entirely different argument that that \(K_f\) is hilbertian when \(f\) is PCF and \(K\) is hilbertian, without assuming \(K\) is a number field. From this, we conclude that the hilbertian reduction over \(\wh K_f\) descends to \(K\). In particular, this isolates one point at which more than hilbertianity is required, but suggests a way forward over more general hilbertian fields.

\subsection{Context}

Arboreal representations were introduced by Odoni~\cite{odoni:iterates,odoni:primes,odoni:realising} for \(X=\PP^1_K\), who used them to quantify primes in the orbits: given a rational point \(a\in \PP^1_{\bar K}\), the absolute Galois group of \(K\) acts on the preimages of \(a\) by iterates of \(f\), the roots of \(f^n(x) = a\), which typically form a tree. The Chebotarev density theorem links the group action to divisors of \(f^n(a)\); roughly speaking, ``large'' actions mean there are ``few'' prime divisors. Odoni calculated some of these actions, including the generic case, giving rise to the various Odoni conjectures/questions about the largeness of these representations~\cite{jones:survey}. One can view the \(\ell\)-adic Tate modules of an elliptic curve as special cases of an arboreal representation with \(X\) is an elliptic curve, \(f\) its multiplication-by-\(\ell\) endomorphism, and the identity as a base point, because the preimages of the identity comprise the \(\ell\)-power torsion. From this point of view, Odoni's conjectures are analogous Serre's Open Image Theorem. Since then, some cases of these conjectures have been establish~\cite{BJ:odoni,looper:odoni,specter:odoni}, as well as more general investigations~\cite{PA:abelian,FP:quadratic,FOZ:abelian,AMT:cubicpcf,ABCCF:quadraticpcf} into the structure -- ramification and local properties especially~\cite{AHPW:local.arboreal,BIJJLMRSS:pcf,AHM:finite.ram,epstein:pcf,berger:iterated,ingram:uniformization,HJ:higherramification,adams:sen,adams:nos}.

Profinite iterated monodromy groups (pfIMGs) were introduced by Pink, who studied them in great detail in the quadratic case, calculating the geometric pfIMG in terms of the critical portrait, using this description to quantify the associated constant field extension~\cite{pink:pcf,pink:pci,pink:lift}. The discrete counterpart was used to solve the twisted rabbit problem in complex dynamics~\cite{BN:rabbit}. Less is known about pfIMGs, although the geometric pfIMGs have been calculated in some interesting families~\cite{AH:unicritical,ejder:imgs,BEK:belyi,KNR:compositions,HLW:degreethree}. It turns out that pfIMGs can be realized as a special case of an arboreal representation, but at the same time all arboreal representations can be located within the pfIMG as decomposition subgroups. They are a natural overgroup for a more precise version of Odoni's conjecture: over a number field, are (almost all) decomposition subgroups of a pfIMG open? One might hope that a better understanding of the group theory of pfIMGs would translate into information about arboreal representations. Since the pfIMG has more structure than a typical arboreal representation, one might hope that it is easier to understand. For example, it is a self-similar group, and the self-similarity is induced from the geometry, which cannot happen for most arboreal representations. The geometric pfIMG is equipped with an outer Galois action and in~\cite{AH:unicritical}, the outer action was used to characterize the constant field extension associated to polynomial pfIMGs, an important invariant of the dynamical system -- among other things, the constant field extension appears as a quotient of every arboreal representation.

\subsection{Notation}
For reference, we summarize some notation here. While we largely follow the notation and conventions from~\cite{AH:unicritical} when discussing iterated monodromy groups, we have switched to a ``right-handed'' semidirect product, a corresponding right action on the tree, and a zero-indexing of the labeling.

By pullback, we refer to the normalization of the (usual) fiber product, so that our pullbacks of curves are always smooth.

\begin{enumerate}
    \item \(K\) is a field.
    \item \(t\) is a transcendental over \(K\).
    \item \(\bar K\) is the algebraic closure of \(K\) in \(\bar K(t)\), hence a separable closure of \(K\).
    \item \(\Gamma_K = \Gal(\bar K/K)\) is the absolute galois group of \(K\).
    \item \(\overline{K(t)}\) is a fixed separable closure of \(K(t)\).
    \item \(f\) is a rational function over \(K\) such that the extension \(\bar K(x)\) of \(\bar K(t)\) given by \(f(x) = t\) is tamely ramified,
    \item \(f^{-\infty}(t)\) is the set of preimages of \(t\) by iterates of \(f\) in \(\overline{K(t)}\), including \(t\). In other words, \(f^{-\infty}(t)\) comprises \(t\) and all the roots of the polynomials \(f^n(x) = t\) in some fixed algebraic closure. which is naturally endowed with the structure of a directed graph by \(x\to y\) whenever \(f(x) = y\). By our tameness assumption, this graph is a \(d\)-regular tree rooted at \(t\).
    \item \(\IMG\) is the arithmetic profinite iterated monodromy group associated to \(f\) over \(K\), the galois group of \(K(f^{-\infty}(t))/K(t))\)
    \item \(\bIMG\) is the geometric  profinite iterated monodromy group associated to \(f\), the galois group of \(\bar K(f^{-\infty}(t))/\bar K(t))\), identified with a subgroup of \(\IMG\) by the natural restriction
    \item \(\Lambda\) is a choice of path data  as in~\cite[Section 3]{AH:unicritical} but numbered from \(0\) to \(d-1\) rather than \(1\) to \(d\).
    \item \(T_{f,\Lambda}\) is the tree of preimages \(f^{-\infty}(t)\) equipped with the labeling induced by \(\Lambda\) as in~\cite[Section 3]{AH:unicritical}. One can identify \(T_{f,\Lambda}\) with its set of branches, or equivalently with left-infinite words in the alphabet \(\{0,1,...,d-1\}\).
    \item \(\Aut T_{f,\Lambda}\) is the set of automorphisms of \(T_{f,\Lambda}\) which respect its tree structure. With respect to the labeling, it has a self-similar description as
    \[\Aut T_{f,\Lambda} \cong \Aut T_{f,\Lambda} \wr S_d = (\Aut T_{f,\Lambda})^d \rtimes S_d,\]
    so \(g\in \Aut T_{f,\Lambda}\) can be written in ``right-handed'' form as
    \[(g_0,...,g_{d-1}) \sigma_g\]
    where \(g_i \in \Aut T_{f,\Lambda}\) and \(\sigma_g \in S_d\). The right action on a left-infinite word \(w\) is as follows: write \(w=vx\) with \(v\) a left-infinite word and \(x\in \{0,1,...,d-1\}\), then
    \[
        (vx)^g = v^{g_x} x^{\sigma_g}.
    \]
    \item \(\odometer\) is the standard odometer in \(\Aut T_{f,\Lambda}\), and \(\Odometer\) is the topological closure of the subgroup it generates. Letting \(\sigma=(0\,1\,...\,d-1)\) denote the standard \(d\)-cycle, the standard odometer is defined by the wreath recursion
    \[\odometer = (1,...,\odometer) \sigma.\]
    If we identify words with \(d\)-adic integers by
    \[...a_2a_1a_0 \longmapsto \sum_{n=0}^\infty a_nd^n,\] the action of the standard odometer coincides with adding \(1\); the carry is precisely the self-similar reappearance of the odometer in its defining wreath recursion.
    \item \(\Arb\) and \(\bArb\) are the images of \(\IMG\) and \(\bIMG\) in  subgroups of the automorphism group of \(T_{f,\Lambda}\).
\end{enumerate}

Let \(G\) be a \(d\)-regular self-similar group, which comes with a tree \(T\), \(B\) the set of branches, and \(\Lambda\) the paths realizing the self-similarity. Let \(v\) be a vertex of \(T\).
\begin{enumerate}
    \item \(T_v\) is the subtree of \(T\) lying over \(v\), and given an integer \(n\), the \(T_{(n)}\) is the union of \(T_w\) as \(w\) ranges over the vertices at level \(n\),
    \item \(B_v\) is the set of branches passing through \(v\),
    \item \(G_v\) is the stabilizer of \(v\) in \(G\), and given an integer \(n\), the level \(n\)-stabilizer \(G_{(n)}\) is the intersection of \(G_w\) as \(w\) ranges over the vertices at level \(n\),
    \item \(\lambda_v\) is the path in \(\Lambda\) associated to \(v\),
    \item \(\pi_v:G_v\to G\) is the map from \(G_v\) to \(G\) induced by \(\lambda_v\),
    \item \(G^v\) is the image of \(\lambda_v\) and for any vertex \(w\) at level \(n\), the image of \(\lambda_w\) restricted to \(G_{(n)}\) is denoted \(G^{(n),w}\). If \(G\) acts transitively on level \(n\), \(G^{(n),w}\) does not depend on \(w\), and we simply write \(G^{(n)}\).
\end{enumerate}
Our paths go in the opposite direction of those in~\cite{AH:unicritical}, one convention is simply the inverse of the other. The choice amounts to deciding whether \(\lambda_v\) and \(\pi_v\) should be parallel or not. The latter is convenient for composition and labeling, whereas the former is more suitable for our upcoming discussion of functoriality.

\section{Preliminaries}

We briefly recall the notion of a self-similar group given by a self-similar action on a tree and their basic features, mainly in the special case of profinite iterated monodromy groups. For general self-similar groups, a standard reference is~\cite{nekrashevych:ssgs}, though in some cases, our definitions are narrower. Facts about profinite iterated monodromy groups are drawn from Section 3 of \cite{AH:unicritical}, and Section 2 of the same also contains a summary of self-similar groups. 

Self-similar groups arise from self-similar actions on self-similar sets. In our case, the actions are on rooted regular trees.
\begin{definition}
    The \(d\)-regular rooted tree of height \(n\), denoted \(T(d^n)\) is the directed rooted tree inductively constructed as follows. When \(n=0\), we let \(T_{d^0}\) be a graph with one vertex, its root, and no edges. For \(n > 1\), we construct \(T(d^n)\) from \(T(d^{n-1})\) by adjoining \(d^n\) new vertices by adding \(d\) new vertices for each leaf \(v\) of \(T_{d^{n-1}}\) and with arrows from each to \(v\).

    When \(m\leq n\), the construction gives a natural graph injections from \(T(d^m)\) into \(T(d^n)\), and the direct limit is a directed graph denoted \(T(d^\infty)\). Any directed graph isomorphic to \(T(d^\infty)\) is called \emph{an infinite \(d\)-regular rooted tree}. A given \(T(d^n)\) is referred to as its \emph{level \(n\)} truncation, whose leaves comprise the \emph{vertices at level \(n\)}. There are also natural graph surjections from \(T(d^n)\) to \(T(d^m)\) given by deleting the vertices and edges of \(T(d^n)\) which are not in \(T(d^m)\), whose inverse limit \(B(d^\infty)\) the set of branches of \(T(d^\infty\).
\end{definition}

Edge-preserving maps between infinite \(d\)-regular rooted trees correspond to continuous endomorphisms of their branch space, so the two objects are broadly equivalent.
\begin{definition}
    Let \(T\) be an infinite \(d\)-regular rooted tree, and \(B\) its set of branches. Given a vertex \(v\) of \(T\), we define some subsets of each space:
    \begin{enumerate}
        \item \(T_v\) is the subtree over vertices of \(T\) with a path to \(v\),
        \item \(B_v\) is the set of branches in \(B\) which pass through \(v\), and can be identified with the set of branches of \(T_v\),
        \item \(T_{(n)}\) is the union of all \(T_v\) as \(v\) ranges over the vertices at level \(n\),
        \item \(B_{(n)}\) is the union of all sets \(B_v\) as \(v\) ranges over the vertices at level \(n\).
    \end{enumerate}
    The sets \(B_v\) are open (and closed) in \(B\) and form a basis for its topology. 
\end{definition}

These sets realize the self-similarity: \(T_v\) is itself an infinite \(d\)-regular rooted tree and its branches can be identified with \(B_v\). How one identifies \(T\) and \(T_v\) or \(B\) and \(B_v\) is not necessarily unique, though one determines the other. In our setting, the trees and branches arise from the fibers of covers, so the self-similarities are induced by (\'etale) paths.
\begin{definition}
    Let \(T\) be the infinite \(d\)-regular rooted tree with root \(t\). Given a vertex \(v\), an \emph{(\'etale) path from \(v\) to \(t\)} is a homeomorphism \(\lambda: T_v \to T\). A \emph{complete set of (\'etale) paths} is a choice of path \(\lambda_v\) from each vertex \(v\).

    Given paths \(\lambda_v\) and \(\lambda_w\) from vertices \(v\) and \(w\), respectively, their composition \(\lambda_w\circ \lambda_v\) is a path from \(u = \lambda_v\inv \circ \lambda_w\inv(t)\) to \(t\). We say that a complete set of \'etale paths is \emph{consistent} if \(\lambda_w \circ \lambda_v = \lambda_u\).
\end{definition}

A complete and consistent set of paths is determined by the paths from vertices just at the first level -- with respect to the original topological picture, this comes from iteratively lifting the paths over each other.

\begin{proposition}\label{proposition: unordered labeling}
    Let \(T\) be an infinite \(d\)-regular rooted tree. For each vertex \(v\) at level 1, choose a path \(\lambda_v:T_v\to T\). For each vertex \(w\) at level \(n\) there is a unique choice of level one vertices \(w_i\) such that the composition
    \[\lambda_{w_n}\circ ... \circ \lambda_{w_1}\]
    is a path from \(w\) to \(t\).
\end{proposition}
\begin{proof}
    Such a composition is only defined at \(w\) if and only if \(w_1\) lies under \(w\), and there is precisely one such vertex at level \(1\). Then \(\lambda_{w_1}(w)\) is one level lower, so by induction we obtain the remaining \(w_2,...,w_n\).
\end{proof}
This construction allows us to view any vertex \(w\) of the tree as being labeled by vertices at level \(n\), once paths are chosen. One can then label the level one vertices by \(\{0,1,...,d-1\}\), which amounts to ordering them, and thereby obtain a consistent labeling of the entire tree by words in those letters.
\begin{definition}
    Let \(T\) be an infinite \(d\)-regular rooted tree. A \emph{choice of path data \(\Lambda\) for \(T\)} is a bijection between \(\{0,1,...,d-1\}\) and the vertices at level \(1\), and a choice of \(d\) paths \(\lambda_v:T_v\to T\), one for each vertex \(v\) on level one.
\end{definition}

A choice of path data \(\Lambda\) allows us to express group actions on the tree ``in coordinates''. Namely, there is an isomorphism
    \[\Lambda: W \overset\sim\to W^d \rtimes S_d\]
which can be given explicitly as follows. Let \(\sigma_g\) denote the restriction of \(g\) to level one, viewed as a permutation in \(S_d\) by way of the ordering. Notice \(g \sigma_g\inv\) fixes level one of the tree, and therefore acts separately on each of the \(d\) copies of the infinite \(d\)-regular rooted tree over those vertices. To be precise, let \(\lambda_i\) denote the path from the vertex labeled by \(i\) to \(t\). Then we define ``coordinates'' by
\[w_i = \lambda_i \circ (g\sigma_g\inv) \circ \lambda_i\inv.\]
The \(g_i\) are in \(W\), and correspond to the action of \(g\sigma_g\inv\) on the tree over the \(i\)th vertex when that tree is identified with the original tree by the path \(\lambda_i\). Finally, we define
\[\Lambda(g) = (g_0,...,g_{d-1}) \sigma_g.\]
The natural right action on left-infinite sequences of digits in \(\{0,1,...,d-1\}\) is given by 
\[(vx)^g = (vx)^{(g_0,...,g_{d-1})\sigma_g} = v^{g_x} x^{\sigma_g}.\]

The map \(\pi_i\) given by \(\pi_i(g)=g_i\) is a homomorphism when restricted to \(W_i\), the stabilizer of the \(i\)th vertex. Iteratively applied to the paths given by Proposition~\ref{proposition: unordered labeling}, we obtain for every vertex \(v\) of \(T\) a map \(\pi_v\) from the stabilizer \(W_v\) of \(v\) to \(W\). These are call the \emph{self-similarity} maps, or coordinate maps, .

\begin{definition}
    Let \(T\) be an infinite \(d\)-regular rooted tree, with \(\Lambda\) a choice of path data. Let \(G\) be a group which acts faithfully on \(T\), viewed as an inclusion of \(G\) into \(\Aut T\). Then \(G\) is \emph{self-similar with respect to \(\Lambda\)} if, for every vertex \(v\) of \(T\), the image of the restriction of \(\pi_v\) to the stabilizer \(G_v\) is contained in \(G\).

    If \(G\) is self-similar, we denote the image \(\pi_v(G_v)\) by \(G^v\). If \(G^v = G\), then we say that \(G\) is \emph{fractal with respect to \(\Lambda\)}.
\end{definition}
By construction, to determine whether \(G\) is self-similar (resp. fractal), it suffices to check that \(\pi_v(G_v)\subseteq G\) (resp. \(\pi_v(G_v) = G\) just for the vertices \(v\) on level one.

Finally, we recall how profinite iterated monodromy groups can be interpreted as self-similar groups.
\begin{theorem}[{\cite[Section 3]{AH:unicritical}}]
    Let \(f\) be a tamely ramified rational map over \(K\) and \(t\) a transcendental over \(K\). The set \(T\) of all preimages of \(t\) by iterates of \(f\), written \(f^{-\infty}(t)\), naturally has the structure of an infinite \(d\)-regular tree rooted at \(t\), with edges \(x\to y\) when \(f(x) = y\). The Galois group \(\IMG = \Gal(K(f^{-\infty}(t))/K(t))\) acts faithfully and transitively on \(T\). Upon picking a common algebraic closure, \(\bIMG = \Gal(K(f^{-\infty}(t))/K(t))\) can be viewed as a subgroup of \(\IMG\).
    
    The level one vertices are the roots of \(f(X) - t\) over \(K(t)\), so let \(t_0,...,t_{d-1}\) be an ordering of them. Any choice of field isomorphisms \(\lambda_i: \bar K(f^{-\infty}(t))\to \bar K(f^{-\infty}(t))\) taking \(t_i\) to \(t\) induces paths from \(t_i\) to \(t\). This gives a choice of path data \(\Lambda\), which induces an inclusion of \(\bIMG\) and \(\IMG\) into \(\Aut W\) such their respective images \(\bArb\) and \(\Arb\) are self-similar and fractal with respect to \(\Lambda\).
\end{theorem}
\begin{proof}
    Only the claim that \(\bArb\) and \(\Arb\) are fractal does not appear directly in~\cite{AH:unicritical} is that \(\bArb\) and \(\Arb\) are fractal, rather than merely self-similar. 

    Let \(g \in \Arb\). Given a coordinate \(i\), we seek some \(\tilde g \in \Arb\) such that \(\tilde g = (...,g,...)\tau\) where \(g\) appears in the \(i\)th coordinate and \(\tau\) fixes \(i\). Observe that \(\Gal(K(f^{-\infty}(t_i))/K(t_i)\) is isomorphic to \(\Arb\) by way of the path from \(t_i\) to \(t\), so that it has an element \(g'\) which acts as \(g\) on the tree over \(t_i\). By the lifting property for normal extensions, we can extend \(g'\) to an element \(\tilde g\) of \(\Gal(K(f^{-\infty}(t))/K(t)\) which acts as \(g\) when restricted to \(K(f^{-\infty}(t_i))\). This lift necessarily fixes \(t_i\), meaning \(\sigma_{\tilde g}\) fixes \(i\), and by construction, the \(i\)th coordinate is \(g\).
\end{proof}

\section{Semiconjugacy and Subgroups of pfIMGs}\label{section: semiconjugacy}

In this section, we show that self-similar subgroups of profinite iterated monodromy groups correspond to a particular kind of semiconjugacy. Then, we show that these can all be obtained as quotients of semiconjugacies by \emph{normal} covers. Using this, we partly classify these exceptional relationships.

\subsection{Induced Semiconjugacies}

\begin{theorem}
\label{theorem: self-similar functoriality}
    Let \(K\) be a field, \(f\) a tamely ramified rational function over \(K\), and \(t\) a transcendental over \(K\). Let \(\Arb\) be the (arithmetic) profinite iterated monodromy group of \(f\) over \(K\) associated to some path data. Let \(\lambda_0:K(t_0)\to K(t)\) be the field isomorphism for the first coordinate, where \(t_0=\lambda_0\inv(t)\) is the associated preimage of \(t\).

    If \(H\) an open self-similar subgroup of \(\Arb\), then \(H\) is fractal. Let \(L_0\) and \(L\) be the fixed fields of \(H_0\) and \(H\), respectively, in \(K(f^{-\infty}(t))/ K(t)\). Then \(L_0\) contains \(L\), and the self-similarity field isomorphism \(\lambda_0\) takes \(L_0/ K(t_0)\) isomorphically onto \(L/K(t)\). Moreover, \(L_0 = LK(t_0) = L(t_0)\) and \(L\) is linearly disjoint from \(K(t_0)\) over \(K(t)\). Additionally, for all sufficiently large \(n\), the field \(K(f^{-n}(t))\) contains \(L\).

    Geometrically, there is a smooth integral curve \(Y_H\) with maps \(f_H:Y_H\to Y_H\) and \(\pi_H:Y_H\to \PP^1_{K}\) such that the following diagram commutes and is a pullback
    \[\begin{tikzcd}
    	{Y_H} & {Y_H} \\
    	{\PP^1_K} & {\PP^1_K}
    	\arrow["{f_H}", from=1-1, to=1-2]
    	\arrow["{\pi_H}"', from=1-1, to=2-1]
    	\arrow["{\pi_H}", from=1-2, to=2-2]
    	\arrow["f"', from=2-1, to=2-2]
    \end{tikzcd}\]
    Along with maps \(\eta_{H,n}: X_{(n)} \to Y_H\). The profinite iterated monodromy group of \(f_H\) can be identified with the subgroup \(H\).

    Conversely, if \(L/K(t)\) is a finite tamely ramified cover of \(K(t)\) linearly disjoint from \(K(t_0)\) over \(K(t)\) such that \(L(t_0)\) over \(K(t_0)\) is isomorphic to \(L\) over \(K(t)\), and \(L\) is contained in \(K(f^{-n}(t))\) for some \(n\). Let \(H\) be the subgroup of \(\Arb\) fixing \(L\). Then \(H\) is self-similar as a subgroup of \(\Arb\).

    Geometrically, if there is a pullback square as above and a positive integer \(n\) such that \(X_{(n)} \to Y_H\), then the path data for \(f\) lifts to path data for \(g\) such that \(\Arb_g\) is a self-similar subgroup of \(\Arb_f\).
\end{theorem}
\begin{proof}
    By definition, \(H_0\subseteq H\), so the galois correspondence for fixed fields ensures \(L_0\supseteq L\). The isomorphism of fields \(\lambda_0\) is translated by galois theory into the self-similarity projection \(\pi_0:\Arb_0 \to \Arb\).

    Let \(L'\) be the fixed field of \(\pi_0(L_0)\). By self-similarity, \(\pi_0(L_0) \subseteq L\), hence \(L' \supseteq L\). That requires \([L':K(t)] \geq [L:K(t)]\), with equality if and only if \(L=L'\). By the galois correspondence, the degree \([L':K(t)]\) is \([G:\pi_0(L_0) ]\), which will furnish the reverse inequality:
    \begin{align*}
        [G:\pi_0(L_0)]
        &= [\pi_0(G_0):\pi_0(L_0)] &\textrm{pfIMGs are fractal,}\\
        &\leq [G_0:L_0] & \textrm{homomorphism decreases index,}\\
        &\leq [G:L] & \textrm{intersection decreases index,}
    \end{align*}
    Therefore, the inequalities must all be equalities and \(\lambda_0(L_0) = L\). Since \(L_0 = \Arb_0\cap H\), the galois correspondence says \(L_0 = LK(t_0) = L(t_0)\). The collapse of the inequalities into equalities also yields \([G_0:L_0] = [G:L]\), hence \([L_0:K(t_0)] = [L:K(t)]\), which is equivalent to the linear disjointness of \(L\) and \(K(t_0)\) over \(K(t)\).

    Since the subgroups \(\Arb_{(n)}\) form a neighborhood basis around the identity with trivial intersection and \(H\) is open, it must be that \(H\) contains \(\Arb_{(n)}\) for all sufficiently large \(n\); again by the galois correspondence, \(K(f^{(-n)}(t))\) contains \(L\) for such \(n\).
    
    The field structures are translated into geometry by the usual correspondence of function fields and smooth curves. That \(\lambda_0\) induces an isomorphism of the covers, taking the coordinate \(t_0\) to \(t\) ensures that it is the \emph{same} cover \(\pi_H:X_L \to \PP^1_K\) on both sides. The covers \(q_{H,n}:X_{(n)}\to Y_H\) are induced by the inclusion \(\Arb_{(n)}\subseteq H\).

    Finally, we verify the self-similarity of \(H\). Let \(L(t_0) = K(t_0,y_0)\) and \(L=K(t,y)\) where \(y_0\) and \(y\) satisfy \(\pi_H(y_0) = t\) and \(f_H(y_0) = y\), and also define \(y_v = \lambda_v\inv y_0(y_0)\). The construction ensures that \(y_v\) is a root of \(f_H(Y) = y\) which satisfies \(\pi_H(y_v) = t_v\). Therefore, the \(y_v\) are all distinct roots of \(f_H(Y) = y\) and \(\lambda_v(y_v) = y\), so that the original path data for \(f\), restricted to the pfIMG field extension for \(f_H\) and \(y\), forms a set of path data and hence gives rise to compatible labelings and tree actions for \(H\) as a subgroup of \(\Arb\).

    The converse essentially follows the same argument in reverse. The key facts are that linear disjointness means inclusions of \(L(t_0)\) into an algebraic closure over \(K(t)\) are in bijection with pairs of inclusions from \(K(t_0)\) and \(L\) into an algebraic closure over \(K(t)\), and that if \(L\subseteq K(f^{-n}(t))\) and \(L\cong L(t_0)\), then \(L(t_0) \subseteq K(f^{-n}(t_0)) \subseteq K(f^{-(n+1)}(t))\) so that the maps \(\lambda_v\) are defined on \(L(t_0)\). Writing \(L(t_0) = K(t,t_0)\) and \(L = K(t,y)\) where \(g(y_0) = y\) and \(h(y) = t\) with degrees \([L(t_0):L]\) and  \([L:K(t)]\), respectively, we see by counting that \(\lambda_v\inv (y) = y_v\) are distinct roots of \(g(Y) = y\) that satisfy \(h(y_v) = t_v\). This means that the paths chosen for \(f\) are also a collection of paths for \(g\) and that the labelings are compatible, inducing an inclusion of \(\Arb_g\) as a self-similar subgroup of \(\Arb_f\).

    Inductively, \(f^n\) is linearly disjoint from \(h\) for all \(n\), and lifts to \(g^n\), so the splitting field of \(g^n(Y) = y\) over \(L\) contains the splitting field for \(f^n(T) = t\) over \(k(t)\), and hence \([\Arb_f:\Arb_g] \leq [L:K]\). Since \(\Arb_g\) is profinite, it must be closed, and a closed subgroup of finite index is open.
\end{proof}
\begin{remark}
    A special case of Theorem~\ref{theorem: self-similar functoriality} is base change. The profinite iterated monodromy group for \(f\) over an extension \(K'\) of \(K\) is a self-similar subgroup of \(\Arb\). In that case, \(Y_H = \PP^1_{K'}\), \(f_H=f\), and \(\pi_H\) is the base change from \(\PP^1_{K'} \to \PP^1_K\).
\end{remark}
\begin{corollary}\label{corollary: reduce to galois dynamical pullback}
    In the setting of Theorem~\ref{theorem: self-similar functoriality}, where \(H\) is an open self-similar subgroup of \(\Arb\), let 
    \[N = \bigcap_{a\in\Arb} H^a\]
    be the normal core of \(H\). Then \(N\) is also open and self-similar. Therefore, there is a curve \(Y_N\) and maps \(\pi^N_H: Y_N \to Y_H\) making the following diagram commute.
    \[\begin{tikzcd}
    	{Y_N} & {Y_N} \\
    	{Y_H} & {Y_H} \\
    	{\PP^1_K} & {\PP^1_K}
    	\arrow["{f_N}", from=1-1, to=1-2]
    	\arrow["{\pi^N_H}"', from=1-1, to=2-1]
    	\arrow["{\pi_N}"', curve={height=30pt}, from=1-1, to=3-1]
    	\arrow["{\pi^N_H}", from=1-2, to=2-2]
    	\arrow["{\pi_N}", curve={height=-30pt}, from=1-2, to=3-2]
    	\arrow["{f_H}", from=2-1, to=2-2]
    	\arrow["{\pi_H}"', from=2-1, to=3-1]
    	\arrow["{\pi_H}", from=2-2, to=3-2]
    	\arrow["f"', from=3-1, to=3-2]
    \end{tikzcd}\]
    The squares and overall rectangle are all pullbacks.
\end{corollary}
\begin{proof}
    Since \(H\) is open, it has finitely many conjugates, so \(N\) is also open. By construction, \(N\) is normal in \(\Arb\), so \(N_v\) is normal in \(\Arb_v\), hence \(N^v=\pi_v(N_v)\) is normal in \(\Arb = \pi_v(\Arb_v)\). From \(N\subseteq H\), we have \(N_v\subseteq H_v\), and since \(H\) is self-similar, we get \(N^v \subseteq H^v \subseteq H\). Since \(N\) is, equivalently, the largest normal subgroup of \(\Arb\) contained in \(H\), it follows that \(N^v\) is contained in \(N\), or in other words \(N\) is self-similar.
    
    Then, applying Theorem~\ref{theorem: self-similar functoriality} to \(H\) and \(N\) yields the bottom square and outer rectangle. The inclusions \(N_0\subseteq N_0\) and \(N\subseteq H\) allow the top sides to be completed with possibly-different maps, but because \(N\) and \(H\) are fractal, \(\pi_0\) takes \(N_0\subseteq H_0\) to \(N\subseteq H\); comparing degrees, or indices, it follows that \(\lambda_0\) takes the cover corresponding to \(N_0\) over \(H_0\) isomorphically to that of \(N\) over \(H\), so the covers coincide, furnishing one map \(\pi^N_H\) that completes the diagram.
\end{proof}

\subsection{Dynamical Pullbacks}
Recall that the Chebyshev and Latt\`es maps are rational functions which arise as quotients of endomorphisms of \(\GG_m\) or an elliptic curve, respectively. It is not difficult to see that these are both special cases of the kind of quotient appearing in Corollary~\ref{corollary: reduce to galois dynamical pullback}. At a first glance, the corollary produces a semiconjugacy between \(f\) and \(f_H\), but the semiconjugacy is of a special and rather ridgid form. In light of this, we define the following family of rational maps which generalizes the Chebyshev and Latt\`es constructions. Those maps are often exceptional in dynamics, and this larger family is exceptional in the results of the present paper. They can be adequately controlled by Riemann-Hurwitz and branch-cycle arguments when \(f\) is a polynomial, but their structure when \(f\) is a general rational function seems mysterious, though Riemann-Hurwitz still imposes serious restrictions on their ramification.

\begin{definition}\label{definition: dynamical pullback}
    Let \(f\) be a tame endomorphism of \(\PP^1_{\bar K}\). Let \(Y\) be an irreducible smooth curve with a cover \(\pi: Y\to \PP^1_{\bar K}\). We say that \(f\) has a \emph{dynamical pullback along \(\pi\)} if there exists an endomorphism \(f_\pi\) of \(Y\) such that the following diagram commutes and is a pullback:
    \[\begin{tikzcd}
    	Y & Y \\
    	{\PP^1_{\bar K}} & {\PP^1_{\bar K}}
    	\arrow["{f_\pi}", from=1-1, to=1-2]
    	\arrow["\pi"', from=1-1, to=2-1]
    	\arrow["\pi", from=1-2, to=2-2]
    	\arrow["f"', from=2-1, to=2-2]
    \end{tikzcd}\]
   We say that a dynamical pullback is galois if the cover \(\pi\) is galois, strict if \(\deg \pi > 1\), and has genus \(g\) when the genus of \(Y\) is \(g\). If \(f\) is a rational map for which there exist \(Y\) and \(\pi\) that, with \(f\), realize a dynamical pullback, we say that \(f\) is a quotient of a dynamical pullback (or dynamical pullback quotient). For brevity, we often omit ``dynamical'' and ``pullback'', simply referring to \(f\) as a quotient.
   
   We say that \(f\) is a strict quotient if it can be realized as the quotient of a strict dynamical pullback; then we say it is galois if there is such a galois dynamical pullback, and genus \(g\) if there is a genus \(g\) dynamical pullback.

   Finally, we say that a dynamical pullback is \emph{induced} if there is an \(n\) such that \(Y\) is a sub-cover of the galois closure of \(f^n:\PP^1_{\bar K}\to\PP^1_{\bar K}\), or equivalently, if the function field of \(Y\) as an extension of \(\bar K(t)\) given by \(\pi(y) = t\) is a subfield of \(\bar K(f^{-n}(t)\). In this situation, we say that \(f\) is a strict induced quotient.
\end{definition}
\begin{remark}
    In field-theoretic terms, pullback means that the extensions of \(\bar K(t)\) corresponding to \(f\) and \(\pi\) in \(\overline{K(t)}\) are linearly disjoint over \(\bar K\). If \(\pi\) is assumed to be galois, it suffices for the fields to intersect trivially. Equivalently, the normalization of the curve
    \[\pi(y) = f(x)\]
    is geometrically integral.
\end{remark}

\begin{remark}
    Dynamical pullbacks are semiconjugacies. However, not all semiconjugacies are dynamical pullbacks. Consider the trivial semiconjugacy \(f\circ f = f\circ f\), with associated square
    \[\begin{tikzcd}
    	{\PP^1_{\bar K}} & {\PP^1_{\bar K}} \\
    	{\PP^1_{\bar K}} & {\PP^1_{\bar K}}
    	\arrow["f", from=1-1, to=1-2]
    	\arrow["f"', from=1-1, to=2-1]
    	\arrow["f", from=1-2, to=2-2]
    	\arrow["f"', from=2-1, to=2-2]
    \end{tikzcd}\]
    This is not a pullback. From the algebraic point of view, the failure in this case is because the field extensions for \(f\) and \(h\) will coincide, hence certainly fail to be linearly disjoint. Geometrically, this is because the actual pullback of two copies of \(f:\PP^1_{\bar K}\to\PP^1_{\bar K}\) is not irreducible; it is the normalization of the fiber product, given in coordinates by \(f(x) = f(y)\), which has a component \(x=y\).
\end{remark}

The Chebyshev map is the dynamical pullback arising from \(Y = \PP^1_{\bar K}\), \(f_\pi(y) = y^d\), and \(\pi(y) = y + \frac 1 y\). In fact, it is dynamically induced immediate at \(n=1\). Latt\`es maps are also dynamical pullbacks. A slightly less common family of exceptions are the twisted Chebyshev polynomials.

\begin{definition}
    Let \(d\) be a positive integer prime to \(\cchar K\). A rational function \(f\) of degree \(d\) is said to be a \emph{standard twisted Chebyshev polynomial} if there is some \(\zeta\) such \(\zeta^{d-1} = 1\) and \(f\) satisfies the identity
    \[f\left(x + \frac 1 x\right) = x^d + \frac \zeta {x^d}.\]
    Any rational function conjugate to the standard twisted Chebyshev map is called a twisted Chebyshev. 
\end{definition}
\begin{proposition}
    A standard twisted Chebyshev map is a polynomial.
\end{proposition}
\begin{proof}
    Let \(\pi(x) = x + \frac 1x\) and \(\pi_\zeta(x) = x + \frac \zeta x\), and \(g(x) = x^d\). The twisted Chebyshev maps are determined by the semiconjugacy
    \[f\circ \pi = \pi_\zeta \circ g.\]
    Note that \(g\) fixes \(0\) and \(\infty\), while \(\pi\) and \(\pi_\zeta\) satisfy
    \[\pi(0)=\pi(\infty)=\pi_\zeta(0) =\pi_\zeta(\infty) = \infty.\]
    Therefore, \(f\) also fixes \(\infty\). Moreover, \(g\) is totally tamely ramified over both \(0\) and \(\infty\), so \(f\) totally tamely ramified over \(\infty\). The rational functions with \(\infty\) as a totally tamely ramified fixed point are precisely the polynomials.
\end{proof}

It turns out that there are not too many more families of strict dynamical pullbacks. We can start to classify them in terms of \(Y\) and the galois group of \(\pi\). 

\begin{proposition}\label{proposition: dynamical pullback options}
    Suppose \(f\) is a dynamical pullback quotient realized by some fixed \(Y\), \(\pi\), and \(f_\pi\) as in Definition~\ref{definition: dynamical pullback}.
    
    If \(\deg f > 1\), then the genus \(g\) of this pullback is \(0\) or \(1\). When the pullback is also galois:
    \begin{enumalpha}
        \item If \(g = 0\), then \(Y\) is a projective line, and \(\Gal(\pi:Y\to\PP^1_{\bar K})\) is a finite group of symmetries of it, hence a finite subgroup of \(\PGL_2\). Therefore, the galois group is cyclic, dihedral, \(A_4\), \(S_4\), or \(A_5\). Moreover, \(\pi\) is ramified over at most \(3\) places and the triple of ramification indices over the branch points determines the group up to isomorphism.
        \item If \(g = 1\), then \(Y\) is an elliptic curve. For the same reasons as in the genus zero case, \(\pi\) is a finite group of symmetries of this elliptic curve, and \(f\) is Latt\`es.
    \end{enumalpha}
\end{proposition}
\begin{proof}
    Since \(f\circ \pi = \pi \circ f_\pi\), we have \(\deg f_\pi = \deg f > 1\). A curve of genus \(2\) or greater admits no endomorphisms of degree strictly greater than \(1\), leaving only genus \(0\) and \(1\) as possibilities for the genus of \(Y\).

    If \(g_Y = 0\) then \(Y\) is a projective line, so \(\Gal(Y/\PP^1_{\bar K})\) is a finite subgroup of the symmetries of the projective line. The conclusion  summarizes Klein's classification; see \cite[Theorem I.6.2]{MM:inversegalois} for an outline.

    If \(g_Y = 1\) then \(Y\) is (over \(\bar K\)) an elliptic curve, and so \(f\) is Latt\`es.
\end{proof}

The previous section showed that open self-similar subgroups of \(\bArb\) give rise to induced quotients, so the present geometric considerations place limits on when it is possible for such a subgroup to exist. In general, Riemann-Hurwitz places substantial limits on dynamical pullback quotients. When \(f\) is a polynomial, only twisted Chebyshev quotients can appear.
\begin{theorem}\label{theorem: classification of open self similar pfIMG quotients}
    Let \(\bArb\) be the geometric profinite iterated monodromy group of \(f\). Suppose \(H\) is an open self-similar subgroup of \(\Arb\) and let \(N\) be the normal core of \(H\). Let \(L\) be the fixed field of \(N\) and \(Y_N\) the corresponding curve with cover \(\pi_N:Y_N\to \PP^1_{\bar K}\). Then \(f\) is an induced pullback quotient, where \(Y=Y_N\) and \(\pi = \pi_N\).

    In particular, \(\bArb/N\) is either a finite subgroup of automorphisms of an elliptic curve and \(f\) is Latt\`es, or \(\bArb/N\) is one of the finite subgroups of \(\PGL_2\): cyclic, dihedral, \(A_4\), \(S_4\), or \(A_5\). In the latter case, \(\pi_N\) is ramified over two or three points, is determined by those ramification indices, and all three branch points must also be pre-periodic and branch points for \(f\).

    If \(f\) is a polynomial and \(H\neq \bArb\) then \([\bArb:N] = 2\), \(H=N\), \(f\) is a twisted Chebyshev map, \(\bArb\) is the infinite pro-\(d\) dihedral group, and \(H\) is the rotation subgroup, generated by the odometer.
\end{theorem}
\begin{proof}
    If \(H\) is self-similar, then Corollary~\ref{corollary: reduce to galois dynamical pullback} implies that its normal core gives rise to an induced galois quotient. The possibilities for \(\bArb/N\), which are classified by Proposition~\ref{proposition: dynamical pullback options}.

    This leaves the polynomial case. First, recall that there is some integer \(n\) such that \(N\)  contains \(\bArb_{(n)}\) and that \(\odometer^{d^n}\) is in the latter subgroup. From the wreath recursion, it is easy to see that \(\pi_v(\odometer^{d^n}) = \odometer\) for any vertex \(v\) at level \(n\), and so the odometer is in \(N^v\). By self-similarity, the odometer must be in \(N\). By normality, all \(\bArb\)-conjugates of the odometer are in \(N\) as well, and so \(\pi_N\) is unramified over \(\infty\).
    
    This means \(\pi_N\inv (\infty)\) has \(\deg \pi_N\) distinct elements, and \(f_N\) is ramified of degree \(d\) over each member of \(\pi\inv(\infty)\). Since \(\deg f_N = \deg f = d\), we see that \(f_N\) is totally ramified over at least two points. Riemann-Hurwitz then tells us
    \begin{align*}
        2d - 2 
          &= \sum_{P\in X}\sum_{f_N(Q) = P} e_{Q/P} - 1,\\
          &\geq \sum_{P\in \pi\inv(\infty)}\sum_{f_N(Q) = P} e_{Q/P} - 1,\\
          &\geq (\deg \pi_N) (d-1).
    \end{align*}
    which is only satisfiable when \(\deg \pi_N = 2\). The degree of \(\pi_N\) is the index of \(N\) in \(\bArb\). Because \(H\) contains \(N\), this means either \(H=N\) or \(H=\bArb\). The normal core of \(\bArb\) is \(\bArb\), so the former is not possible hence \(H=N\).
    
    \vspace{1em}
    
    Using the ramification information, we can show that \(\bArb\) is the infinite pro-\(d\) dihedral group.
    
    From the Riemann-Hurwitz calculation, the genus of \(Y_N\) is zero, so it is a copy of \(\PP^1_{\bar K}\). Changing coordinates on \(Y_N\) only changes \(f_N\) by conjugation with a Mobi\"us transformation, and \(\pi_N\) by pre-composition with the same transformation. Making such a change if necessary, we may assume \(\pi_N\inv(\infty) = \{0,\infty\}\). Additionally, \(\infty\) is fixed by \(f\), so \(f_N\) either fixes \(0\) and \(\infty\) or swaps them. If they are fixed, then \(f_N(y) = y^d\), and if it swaps them then \(f_N(y) = y^{-d}\). In both cases, \(N\) is procyclic, isomorphic to \(\ZZ_d\), hence abelian, and contains the odometer. The odometer is self-centralizing, so the only abelian subgroup containing it is the subgroup it generates, so \(N = \llangle \odometer \rrangle\). As \(N\) is normal of index \(2\), this gives rise to an exact sequence,
        \[0\to \ZZ_d \to \bArb \to \ZZ/2\ZZ \to 0.\]
    The induced outer \(\ZZ/2\ZZ\) action on \(\ZZ_d\) induces an actual action because the latter is abelian, and because odometer is self-centralizing in \(\bArb\), the action cannot be trivial.
    
    The Riemann Existence Theorem lets us select two elements \(\odometer_0\) and \(\odometer_\infty\), one from an inertia subgroup of \(N\) over \(0\) and \(\infty\), respectively, which generate \(\ZZ_d\) satisfy the product formula \(\odometer_0\odometer_\infty = \id\), hence \(\odometer_\infty = \odometer_0\inv\). Since the monodromy group of \(\pi_N\) exchanges \(0\) and \(\infty\), it must also swap \(\odometer_0\) and \(\odometer_\infty\) up to conjugacy in \(\ZZ_d\); as remarked above, this outer action is an actual action, so \(\ZZ/2\ZZ\) acts as on \(\ZZ_d\) by inversion, and so \(\bArb\) is the pro-\(d\) dihedral group.
    
    \vspace{1em}
    
    Finally, we show that \(f\) is twisted Chebyshev. We have established that \(f_N(y) = y^{\pm d}\) for an appropriate choice of coordinates, and that \(\pi_N\inv(\infty) = \{0,\infty\}\). The latter requires that \(\pi_N\) is of the form    
    \[\pi_N(y) = \frac{ay^2 + by + c}y\]
    for some \(a,b,c\), where both \(a\) and \(c\) are nonzero. Let \(\zeta = c/a\), which is nonzero and finite. Composing with the linear polynomial \(\ell(y) = (y-b)/a\), we obtain
        \[\ell\circ \pi_N (y) = y + \frac {\zeta} y.\]
    Therefore, \(\tilde f = l\inv \circ f\circ \ell\) satisfies
        \[\tilde f \left(x + \frac{\zeta}{x}\right) = x^{\pm d} + \frac{\zeta}{x^{\pm d}}.\]
    Comparing the coefficients of \(x^{\pm d}\), we see that \(\zeta^d = \zeta\). As pointed out above, \(\zeta\) is nonzero, and therefore the conjugate \(\tilde f\) of \(f\) is the standard \(\zeta\)-twisted Chebyshev polynomial of degree \(d\), hence \(f\) is twisted Chebyshev.
\end{proof}

\section{Self-Similar Properties}\label{section: self similar properties}

In this section, we define/construct the self-similar closure and the related notion of self-similar properties of subgroups. Then, we show that many interesting group-theoretic properties are self-similar, especially when the ambient group is fractal. Finally, we combine these with the results of the previous section to show that self-similar properties will expand to the entire group when the associated rational function is not a proper quotient of a dynamical pullback.

\begin{definition}\label{definition: self-similar closure}
    Let \(G\) be a profinite self-similar group and let \(H\) be a closed subgroup of \(G\). The \emph{self-similar closure} of \(H\) is the smallest closed self-similar subgroup \(J\) of \(G\) that contains \(H\), where smallest means that any other closed self-similar subgroup containing \(H\) also contains \(J'\).
\end{definition}
\begin{proposition}
    The self-similar closure exists. It can be constructed by transfinite recursion, and if \(H\) is open, this recursion is a finite process.
\end{proposition}
\begin{proof}
    It suffices to show that arbitrary intersections of closed self-similar subgroups are closed and self-similar. Suppose that \(A(i)\) is a collection of closed self-similar subgroups indexed by \(i\in \mcI\). Let \(A(\infty)\) be the intersection of all the \(A(i)\), which is certainly a closed subgroup. As for self-similarity, we calculate
    \begin{align*}
        A(\infty)^v
          &= \pi_v(A(\infty)_v)\\
          &= \pi_v \left(\bigcap_{i\in \mcI} A(i)_v\right)\\
          &\subseteq \bigcap_{i\in \mcI} \pi_v (A(i)_v)\\
          &= \bigcap_{i\in \mcI} A(i)^v\\
          &\subseteq \bigcap_{i\in \mcI} A(i) \hspace{4em} (A(i)\textrm{ are self-similar})\\
          &= A(\infty).
    \end{align*}
    Therefore, \(A(\infty)^v\subseteq A(\infty)\), so it is self-similar.
    \vspace{1em}

    The explicit construction is as follows. Form an infinite totally ordered list of vertices, each repeated so that every vertex appears arbitrarily high in the list. Using \(\mcI\) to index this list, we form a sequence of subgroups \(H(i)\), where at step \(i+1\), we define \(H(i+1) = H(i) H(i)^{v_i}\). At limit ordinals, we set \(H(i) = \overline{\bigcup_{j < i} H(j)}\).

    The construction produces an ascending chain of subgroups. The repetition of vertices in the list and ensures that for any index \(i\) and vertex \(v\), there is some \(j=j_i\geq i\) such that \(v_j = v\), and hence
    \[H(i)^v \subseteq H(j)^{v_j} \subseteq H(j+1).\]
    Now let \(H(\infty) = \overline{\bigcup H(i)}\) and take a vertex \(v\). Note that \(H(\infty)_v = H(\infty)\cap G_v\) is an intersection of closed sets, the ambient space is Hausdorff and compact, and \(\pi_v\) is continuous, so the fact above yields
    \[H(\infty)^v = \overline{\bigcup_{i\in \mcI} H(i)^v} \subseteq \overline{\bigcup_{j=j_i, i\in\mcI} H(j)^v} = H(\infty).\]
    The final equality follows from \(j\geq i\) and that the sequence is nested. Therefore, \(H(\infty)\) is self-similar. It is clear that each intermediate \(H(i)\) must be contained in any closed self-similar group containing \(H\), and so \(H(\infty)\) is the self-similar closure.

    If \(H\) is open, then it is contained in only finitely many subgroups of \(G\), so the chain \(H(i)\) has only finitely many distinct entries, and the recursion can be refined to a finite one.
\end{proof}

\begin{definition}\label{definition: self-similar property}
    Let \(G\) be a self-similar group. A property \(\mcP\) of closed subgroups of \(G\) is said to be \emph{self-similar} if, whenever \(H\) satisfies \(\mcP\), so too does the self-similar closure of \(H\). For brevity, we say that \(H\) is an \emph{\(\mcP\)-subgroup} when \(H\) satisfies \(\mcP\).
    
    A property of \(\mcP\) of subgroups can be identified with the set of \(\mcP\)-subgroups, and \(\mcP\) is a self-similar property when the set of subgroups of \(G\) satisfying \(\mcP\) is closed under self-similar closure in \(G\), in which case we say that \(\mcP\) or this set is self-similarity-closed.
\end{definition}

The explicit construction of the self-similar closure gives rise to a natural sufficient condition for a property to be self-similar:

\begin{proposition}\label{proposition: construction criteria for self similarity}
    Let \(G\) be a self-similar group and \(\mcP\) a property of its subgroups. Then \(\mcP\) is self-similar if the collection of \(\mcP\) subgroups is closed under the following two operations:
    \begin{enumerate}
        \item Given an \(\mcP\)-subgroup \(H\) and vertex \(v\), the compositum \(HH^v\) is an \(\mcP\)-subgroup.
        \item If \((H_i)_{i\in\mcI}\) is an ascending chain of \(\mcP\)-subgroups, then \(\overline{\bigcup H_i}\) is also an \(\mcP\)-subgroup. 
    \end{enumerate}
    If \(\mcP\) implies open, then the second condition is unnecessary.
\end{proposition}
\begin{proof}
    Immediate from the recursive construction of the self-similar closure: conditions (1) and (2) imply that each step of the recursion produces an \(\mcP\)-subgroup, so by transfinite induction the self-similar closure is also an \(\mcP\)-subgroup. If \(\mcP\) implies open, then every ascending chain is equivalent to a finite chain, because open subgroups are properly contained in only finitely many subgroups.
\end{proof}

\begin{proposition}\label{proposition: odometer is ssc}
    Let \(G\) be a self-similar group which contains the standard odometer \(\odometer\) and the subgroup \(\Odometer\) that it generates. The set of subgroups of \(\Odometer\) is self-similarity-closed.
\end{proposition}
\begin{proof}
    Let \(m = k+du\) be a \(d\)-adic integer, where \(k=m\bmod d\) is an integer between \(0\) and \(d-1\). Then we can write
    \[\odometer^{k+du} = (\odometer^u,...,\odometer^u,\odometer^{u+1},...,\odometer^{u+1})\sigma^k,\]
    from which it follows by induction that the only coordinates of \(\odometer^m\) are again powers of the odometer. Therefore, if \(H\) is a subgroup of \(\Odometer\), so too is \(H^v\), and hence \(HH^v\) is also contained in \(\Odometer\). Since \(\Odometer\) is itself a closed subgroup, ascending unions and topological closures of its subgroups are still contained in \(\Omega\).
\end{proof}

Self-similar properties can and do make reference to how \(H\) is situated within \(G\) (open, closed, normal, ...). In light of the explicit construction, it is not surprising that these properties seem easier to understand when \(G\) is fractal: information about the inclusion \(H\subseteq G\) can pass to \(H_v = H\cap H_v\subseteq G_v\), be transferred by \(\pi_v\) to \(H^v\subseteq G^v = G\), and then finally to the compositum \(HH^v\); in good cases, this will imply \(HH^v\) satisfies \(\mcP\). 

\begin{example}
    Suppose \(G\) is fractal. Then the property ``\(H\) is normal in \(G\)'' is self-similar: \(H \normal G\) implies \(H_v \normal G_v\), hence \(H^v = \pi_v(H_v) \normal \pi_v(G_v) = G\), so \(HH^v\) is also normal in \(G\). Ascending unions and topological closures of normal subgroups are normal as well.
\end{example}

With this in mind, the following proposition allows us to translate some standard group-theoretic properties, like prosolvability or pronilpotence, into equivalent self-similar properties when the ambient group is fractal.

\begin{proposition}\label{proposition: formation implies ssc}
    Let \((G,H)\) denote a pair consisting of a profinite group and subgroup. Suppose a property \(\mcQ\) of such pairs satisfies the following conditions:
    \begin{enumalpha}
        \item If \(\mcQ(G,H)\) and \(K\) is a subgroup of \(G\), then \(\mcQ(K,H\cap K)\).
        \item If \(\mcQ(G,H)\) and \(\pi\) is a homomorphism, \(\mcQ(\pi(G),\pi(H))\).
        \item If \(\mcQ(G,H)\) and \(\mcQ(G,K)\), then \(\mcQ(G,HK).\)
        \item If \((H_i)_{i\in\mcI}\) is an ascending chain of subgroups such that  \(\mcQ(G,H_i)\) for all \(i\in\mcI\), then \(\mcQ\left(G,\overline{\bigcup H_i}\right)\).
    \end{enumalpha}
    If \(G\) is fractal, then the property \(\mcP(H) = \mcQ(G,H)\) is self-similar.

    Also, if \(\mcQ(G,H)\) implies \(H\) is open, then (d) is unnecessary.
\end{proposition}
\begin{proof}
    Assume \(G\) is fractal. Suppose \(\mcQ(G,H)\). Apply (a) to obtain \(\mcQ(G_v,H_v)\), then (b) using \(\pi_v\) implies  \(\mcQ(\pi_v(G_v),\pi_v(H_v))\) hence \(\mcQ(G^v,H^v)\). This furnishes \(\mcQ(G,H^v)\) because \(G\) is fractal. Then we can apply (c) to conclude \(\mcQ(G,HH^v)\). Proposition~\ref{proposition: construction criteria for self similarity} tells us that \(\mcQ(G,-)\) is a self-similar property if either the subgroups \(H\) are open, or they are closed under closure of ascending union, which we have from (d).
\end{proof}
\begin{remark}
    While the criteria of Proposition~\ref{proposition: formation implies ssc} are quite strict, many interesting self-similar properties are induced from more general properties of groups in this way.
\end{remark}
\begin{proposition}\label{proposition: center is ssc}
    Let \(G\) be a fractal group. The set of subgroups of the center is self-similarity-closed.
\end{proposition}
\begin{proof}
    We will verify the four conditions of Proposition~\ref{proposition: formation implies ssc}.
    \begin{enumalpha}
        \item If \(H\) is central in \(G\), all of its subgroups are too, so if \(K\) is any other subgroup, then \(H\cap K\) is central in \(G\), and therefore central in \(K\) too.
        
        \item If \(H\) is central in \(G\), the image of \(H\) remains central in the image of \(G\).
        
        \item If \(H\) and \(K\) are both contained in the center of \(G\), so too is their compositum \(HK\).
        
        \item The center is closed, so it is closed under ascending union and topological closure.
    \end{enumalpha}
\end{proof}

\begin{proposition}\label{proposition: pronilpotent normal is ssc}
    Let \(G\) be a fractal group. The set of open, normal, pronilpotent subgroups is self-similarity closed.
\end{proposition}
\begin{proof}
    We will again use Proposition~\ref{proposition: formation implies ssc}. Normality and openness are self-similar, and conditions (a) and (b) are clearly satisfied even by general pronilpotent subgroups. However, normality is crucial for (c), which in this case is equivalent to Fitting's Theorem~\cite[5.2.8]{robinson:groups}, which says that the compositum of two pronilpotent normal subgroups is again a pronilpotent normal subgroup. 
\end{proof}

\begin{proposition}\label{proposition: induced ssc properties}
    Suppose \(\mcQ\) is a property of (individual) groups which is preserved under passage to subgroups, quotients, and group extension. Then set of subgroups of \(G\) which are open, normal, and satisfy \(\mcQ\) is self-similarity closed.
\end{proposition}
\begin{proof}
    We will verify the three conditions of Proposition~\ref{proposition: formation implies ssc}. As remarked previously, normality and openness are self-similar in fractal groups, so we focus on \(\mcQ\); as in the pronilpotent case, we will make use of normality in checking the third condition of the critierion.
    
    Since \(\mcQ\) is preserved by passing to subgroups, the intersection \(H\cap K\) satisfies \(\mcQ\), hence (a) holds, which further implies (b) because \(\mcQ\) also respects quotients.

    As for (c), normality of \(H\) allows us to write \(HK\) as an extension:
\[0\longrightarrow H \longrightarrow HK \longrightarrow \frac{HK}{H} \longrightarrow 0\]
    The second isomorphism theorem can be applied because \(K\) is normal, hence
    \[\frac{HK}{H}\cong \frac{K}{H\cap K},\]
    and the right hand side is a quotient of \(K\), hence satisfies \(\mcQ\).
\end{proof}
\begin{corollary}\label{corollary: induced ssc examples}
    The following families of open normal subgroups of \(G\) are self-similarity closed:
    \begin{enumerate}
        \item Open, normal, and torsion.
        \item Open, normal, and pro-\(p\).
        \item Open, normal, and prosolvable.
    \end{enumerate}
\end{corollary}

\section{Main Results}\label{section: main}

We now combine the geometric considerations of Section~\ref{section: semiconjugacy} with the group-theoretic results from Section~\ref{section: self similar properties} into our main result, Theorem~\ref{theorem: self-similar upgrades}, which shows that self-similar properties tend to expand from open subgroups of the geometric pfIMG to the entire pfIMG. When they do not do so, the failure can be explained terms of Corollary~\ref{corollary: reduce to galois 
dynamical pullback} by the existence of a curve \(Y\) and endomorphism \(g:Y\to Y\) with projections \(\pi:Y\to \PP^1_{\bar K}\) such that \(\bArb_g\) is an open subgroup of \(\bArb_f\) which is maximal for the self-similar property in question.

\begin{theorem}\label{theorem: self-similar upgrades}
    Suppose that \(f\) is a tamely ramified endomorphism of \(\PP^1_{\bar K}\) and \(\bArb\) its geometric pfIMG. Let \(\mcP\) be a self-similar property. 

    If \(\bArb\) is virtually \(\mcP\), meaning it contains an open \(\mcP\)-subgroup, then either \(\bArb\) itself satisfies \(\mcP\), or \(\bArb\) is not \(\mcP\) and \(f\) is a proper quotient of a dynamical pullback. This proper quotient of a dynamical pullback is induced by an open normal subgroup \(N\) contained in a maximal open \(\mcP\)-subgroup \(M\).

    If \(f\) is a polynomial and not twisted Chebyshev, then \(\bArb\) is virtually \(\mcP\) if and only if it is \(\mcP\). If it is twisted Chebyshev and virtually \(\mcP\) but not \(\mcP\), then the subgroup generated by the odometer is the unique maximal \(\mcP\)-subgroup.
\end{theorem}
\begin{proof}
    If \(\bArb\) has an open subgroup satisfying \(\mcP\), then it has a maximal open subgroup \(M\) satisfying \(\mcP\). The self-similar closure of \(M\) satisfies \(\mcP\) too, so by maximality \(M\) must coincide with its self-similar closure, hence is self-similar. By Theorem~\ref{theorem: classification of open self similar pfIMG quotients}, the normal core \(N\) of \(M\) realizes \(f\) as a quotient of a dynamical pullback, and it is proper precisely when \(N\) is a proper subgroup of \(\bArb\). Since \(N\) is the normal core, \(N=\bArb\) if and only if \(M=\bArb\). If \(N=\bArb\), then \(M=N=\bArb\) is \(\mcP\). Otherwise, \(M\) and \(N\) are proper subgroups of \(\bArb\): the former means \(\bArb\) is not \(\mcP\), by maximality of \(M\), and the latter means that \(f\) is a proper quotient of a dynamical pullback.

    For polynomials, Theorem~\ref{theorem: classification of open self similar pfIMG quotients} showed that the only proper quotients of dynamical pullbacks are the twisted Chebyshev maps, and that the normal subgroup \(N\) is the subgroup generated by the odometer. 
\end{proof}

\begin{corollary}\label{corollary: main results}
    Continue from the assumptions of Theorem~\ref{theorem: self-similar upgrades}. Suppose \(f\) is either a rational function and not a proper quotient of a dynamical pullback, or \(f\) is a polynomial and not twisted Chebyshev.
    \begin{enumerate}
        \item If \(\bArb\) is virtually torsion, then it is torsion.
        \item If \(\bArb\) is virtually pro-\(p\), then \(G\) pro-\(p\).
        \item If \(\bArb\) is virtually prosolvable, then it is prosolvable.
        \item If \(\bArb\) is virtually pronilpotent, then it is pronilpotent.
        \item If \(\bArb\) contains the standard odometer \(\omega\) and \(\llangle \omega\rrangle\) is open, then \(\bArb\) is generated by the odometer.
        \item If \(\bArb\) contains the standard odometer \(\omega\) and the center of \(\bArb\) is open, then \(G\) is abelian.
    \end{enumerate}
    
    In Case (6), \(G\) is procyclic and generated by the odometer.
\end{corollary}
\begin{proof}
    Condition (5) is equivalent to ``\(\llangle \omega\rrangle\) contains an open subgroup'' and Condition (6) is equivalent to ``the center contains an open subgroup''. In each case \(\bArb\) has a \emph{normal} open subgroup witnessing property, because if \(H\) is an open witness, then the (finite) intersection of all conjugates of \(H\) still has the property, remains open, and is normal.

    With that in place, all these properties are self-similar. The references are as follows:
   \begin{itemizeblank}
        \item (1), (2), and (3) in Corollary~\ref{corollary: induced ssc examples}.
        \item (4) in Proposition~\ref{proposition: pronilpotent normal is ssc}.
        \item (5) in Proposition~\ref{proposition: odometer is ssc}.
        \item (6) in Proposition~\ref{proposition: center is ssc}
   \end{itemizeblank}

    For the strengthening of (6), if \(\bArb\) is abelian, every element of \(\bArb\) commutes with the odometer. The odometer, however, is self-centralizing, so this can only happen if \(\bArb\) is contained in the subgroup generated by the odometer.
\end{proof}

\section{Applications}

The main results can be applied to answer or partially answer some questions raised in~\cite[Section 5]{BGJT:specializations} about the Frattini subgroup of profinite iterated monodromy groups. These answers are essentially complete for rational maps which are not proper quotients of a dynamical pullback. In the polynomial case, we saw in Theorem~\ref{theorem: classification of open self similar pfIMG quotients} that only the twisted Chebyshev maps arise, so the only exceptional group is the infinite pro-\(d\) dihedral group, which can be handled separately. 

\subsection{Unicritical Frattini}
In~\cite[Section 5]{BGJT:specializations}, the authors asked whether a unicritical polynomial whose degree is not a prime power could have a profinite iterated monodromy group with an open Frattini subgroup. We answer this in the negative, along the way ruling out the possibility of an open Frattini subgroup for many polynomials, and showing that, typically, open Frattini (more generally, virtually pronilpotent) implies pro-\(p\).

\begin{proposition}\label{proposition: polynomial open Frattini to pronilpotent}
    Let \(f\) be a rational map over \(K\) of degree \(d\) and tamely ramified, and let \(\bArb\) be the geometric profinite iterated monodromy group. Suppose \(f\) is not a proper quotient of a dynamical pullback. Then the Frattini subgroup of \(\bArb\) is open if and only if \(\bArb\) is pronilpotent and finitely generated.

    Suppose further that \(f\) is a polynomial. Then:
    \begin{enumerate}
        \item If \(f\) is a twisted Chebyshev map, then \(\bArb\) is pro-\(d\) dihedral, and hence finitely generated with open Frattini subgroup in all cases, but is only pronilpotent when \(d\) is a power of \(2\).
        \item If \(f\) is not a twisted Chebyshev map, the Frattini subgroup of \(\bArb\) is open if and only if \(\bArb\) is finitely generated and pronilpotent.
    \end{enumerate}
\end{proposition}
\begin{proof}
The Frattini subgroup is pronilpotent, so open Frattini implies virtually pronilpotent, in which case Corollary~\ref{corollary: main results}(4) requires that \(\bArb\) is pronilpotent when \(f\) is not a proper quotient of a dynamical pullback or, if it is a polynomial, that it is not twisted Chebyshev. It only remains to verify the group-theoretic claims about the pro-\(d\) dihedral group.

Let \(r=\rad (d)\) be the product of each prime divisor of \(d\) exactly once, and let \(\sigma\) and \(\tau\) be the rotation and flip, respectively, generating a dihedral group \(D_{2\cdot k}\). The subgroup \(\langle \sigma\rangle\) is maximal. The other maximal subgroups take the form \(\langle \sigma^p,\sigma^i\tau\rangle\) for \(p\) a divisor of \(k\). It follows that the Frattini subgroup is \(\langle \sigma^r\rangle\), which has index \(2r\) in \(D_{2\cdot k}\); in particular, it does not vary among \(D_{2\cdot d^n}\) as \(n\) goes to \(\infty\), so the index of \(\Phi\) in \(\bArb\) is likewise \(2r\).

Moreover, in a pronilpotent group, elements of coprime order must commute. The order of \(\tau\) is \(2\), so \(D_{2\cdot d^\infty}\) can only be pronilpotent if the order of \(\sigma\) in the profinite sense is a power of \(2\), and so \(d\) must be a power of \(2\) as well.
\end{proof}

The easiest way for a group to be pronilpotent is to be pro-\(p\). For polynomials, it turns out that this if often the only way to have a pronilpotent profinite iterated monodromy group.
\begin{proposition}\label{proposition: nilp to pro p criterion}
    Let \(f\) be a polynomial over \(K\) of degree \(d\) and tamely ramified. Make a choice of paths such that \(\bArb\) contains the standard odometer \(\odometer\) as an inertia subgroup generator at \(\infty\). Suppose \(\bArb\)  pronilpotent. If \(\bArb\) contains a nontrivial element of the form \(g = (1,...,1,h,1,...,1)\) then \(d\) is a power of some prime \(p\) and \(\bArb\) is pro-\(p\).
\end{proposition}
\begin{proof}
    Replacing \(g\) with a conjugate by \(\odometer\), we may assume that \(h\) is in the first coordinate.
    
    If \(d\) is not a prime power, then there are profinite integers \(a\) and \(b\) such that: \(b\) is not divisible by \(d\), and the two elements \(g^a\) and \(\odometer^b\) are nontrivial and have coprime orders in the profinite sense. 
    
    Write \(b = k \bmod d\), with \(1\leq k < d\), so
    \[\odometer^b = (\odometer^{c_1},...,\odometer^{c_d})\sigma^k,\]
    where the \(c_i\) are certain exponents determined by \(b\).
    
    Since \(\bArb\) is pronilpotent, it splits into a direct product of its Sylow subgroups, so elements of coprime order commute. Applied to \(g^a\) and \(\odometer^b\), one would obtain:
    \begin{align*}
        (h,1,...,1)
          &= g^a \\
          &= \odometer^{-b}g^a\odometer^b \\
          &= (1,...,1,\odometer^{-c_k} h \odometer^{c_k},1,...)
    \end{align*}
    where the final nonzero coordinate is in index \(1+k \neq 1\), a contradiction.
\end{proof}

\begin{corollary}\label{corollary: open frattini implies pro-p}
    Suppose \(f\) is a polynomial over \(K\) of degree \(d\), tamely ramified, and not conjugate to a twisted Chebyshev map. Let \(\bArb\) be its profinite iterated monodromy group, and assume it contains the standard odometer. Let \(B\) denote the (strict) critical orbit.

    Suppose there is some \(b\) in \(B\) such that \(f\inv (b)\) contains no critical points and \(f\inv(b)\cap B\) has just one element. If \(G\) is virtually pronilpotent, then it is pro-\(p\).
\end{corollary}
\begin{proof}
    Since \(f\) is not twisted Chebyshev, virtual pronilpotent implies \(\bArb\) is pronilpotent by Corollary~\ref{corollary: main results}.
    Moreover, if \(\gamma\) is any generator of an inertia subgroup over \(b\), then~\cite[Lemma 3.12]{AH:unicritical}, show it is of the form
    \[\gamma = (1,...,1,\gamma',1,...,1).\]
    This, by Proposition~\ref{proposition: nilp to pro p criterion}, requires that \(\bArb\) is pro-\(p\).

    The final claim follows from the classical fact that the Frattini subgroup is pronilpotent.
\end{proof}

There are polynomials to which Corollary~\ref{corollary: open frattini implies pro-p} does not apply, but it requires that the branch locus have a somewhat ``dense'' portrait in the sense that every branch point would need to have a preimage which is a critical point, or at least two preimages also in the branch locus. The unicritical polynomials are among those to which the corollary applies.
\begin{corollary}\label{corollary: main special cases}
    If the Frattini subgroup of \(\bArb\) is open in either of these cases:
    \begin{enumerate}
        \item \(f\) is post-critically infinite.
        \item \(f\) is unicritical but not conjugate to a powering map.
    \end{enumerate}
    then \(\bArb\) is pro-\(p\) and \(\deg f\) is a prime power.
\end{corollary}
\begin{proof}
    Note that the twisted Chebyshev polynomials are all PCF, so they cannot satisfy (1), and they're not unicritical except when \(d=2\), in which case the resulting group is pro-\(2\) anyway. When \(f\) is conjugate to a powering map, the iterated monodromy group is procyclic and any degree is possible.

    Now we proceed to verify 
    \begin{enumerate}
        \item Let \(c\) be a critical point with infinite orbit. For sufficiently large \(n\), the point \(b=f^n(c)\) will only have one preimage which is also a branch point, and hence satisfies the assumptions of Corollary~\ref{corollary: open frattini implies pro-p}.
        
        \item If \(f\) is unicritical and the branch locus has at least \(3\) elements, there must be at least one which has only single preimage in the branch locus.
        
        So suppose the branch locus has two elements. In the periodic case, the inertia generator over the critical point satisfies Corollary~\ref{corollary: open frattini implies pro-p}. In the pre-periodic case, the corollary will not apply. Either \(d=2\) and we are in the Chebyshev case that was already handled, or \(d>2\) and we will appeal directly to Proposition~\ref{proposition: nilp to pro p criterion}. The case \(n=2,d>2\) is case (C) of~\cite{AH:unicritical}, and it was shown that
        \[([b_m, [b_m, b_n]],1,...,1) \in \bArb.\]
        While~\cite{AH:unicritical} works with a model of \(\bArb\) that does not change the standard coordinate, one can make a change of paths so that it does. Changing paths only alters the coordinates by conjugation from \(\bIMG\) and permutation, and therefore we will still have an element supported in a single coordinate, to which the proposition can be applied.
    \end{enumerate}
\end{proof}

\subsection{Frattini Specialization}\label{section: specialization}
One of the key insights in~\cite{BGJT:specializations} is that an open Frattini subgroup should allow one to show that the arboreal representation is surjective (in an appropriate sense) over hilbertian fields by reducing to the finite case, in the spirit of Odoni's original questions about these representations. More precisely, they show that if the rational map over a number field is post-critically finite and its \emph{arithmetic} pfIMG is pro-\(p\), then Odoni's conjecture is true: there are infinitely many specializations of the base point with surjective arboreal representation. As we saw in the previous section, open Frattini implies pro-\(p\) for many polynomials, which would allow one to reduce to~\cite{BGJT:specializations}. Nevertheless, the non-\(2\)-power degree twisted Chebyshev case furnish a polynomial examples of an open Frattini subgroup of a geometric pfIMG which is not pro-\(p\), nor even pronilpotent. It seems plausible that rational functions might furnish more exotic examples.  

We will prove that over number fields that if the \emph{geometric} pfIMG has an open Frattini subgroup, then Odoni's conjecture is true. Assumptions about the geometric pfIMG are milder, and in principle easier to verify: we point out that, for PCF rational functions, a pro-\(p\) or even pronilpotent geometric pfIMG implies that the geometric pfIMG has an open Frattini subgroup. For polynomials, we are able to prove the result over all hilbertian fields, further improving the known case from~\cite{BGJT:specializations}. The approach suggests a path toward a specialization result for rational functions over general hilbertian fields.

Our notion of specialization incorporates base change. One could expand and replace \(K\) ahead of time. However, the arithmetic pfIMG is sensitive to base change, so we find it useful to include it in our analysis, and it allows us to articulate where \(K\) is assumed to be a number field, rather than merely hilbertian.
\begin{definition}
    Given \(a\in \bar K\), and \(L\) a subfield of \(\bar K\) containing \(K(a)\) We say that \(\Gal(L(f^{-\infty}(a))/L)\) is a \emph{specialization at \(t=a\) over \(L\)} of the profinite iterated monodromy group. If \(a\) is not post-critical, then \(\Gal(L(f^{-\infty}(a))/L)\) can essentially be identified with a subgroup of \(\IMG\): base change to \(L\) restricts to a subgroup, and so this Galois group can be identified with a decomposition subgroup over \(t-a\), all of which are conjugate. We say that a specialization is surjective if this subgroup is all of \(\IMG\). Typically, one is interested in the case \(a\in K\) so that \(K(a) = K\).
\end{definition}

Then some of main results of~\cite{BGJT:specializations} can be summarized as follows:
\begin{theorem}[\cite{BGJT:specializations}]\label{theorem: BGJT main}
    Let \(f\) be a rational function over \(K\) of degree \(d\), post-critically finite and tamely ramified. Let \(a\) be any element of \(K\) which is not post-critical.

\begin{enumerate}
    \item If \(\Arb\) is a \(p\)-group, then there is an integer \(m\) such that a specialization of the profinite iterated monodromy group at \(t=a\) over \(K\) is surjective if 
    \[\Gal(\bar K(f^{-m}(t))/\bar K(t)) \cong \Gal(\wh K_f(f^{-m}(a))/\wh K_f).\]
    \item If \(K\) is a number field, there are infinitely many \(a\in K\) such that the specialization is surjective.
    \item If \(K\) is a number field and \(f\) is unicritical of prime-power degree, one may take \(m\) to be the size of the critical orbit.
\end{enumerate}
\end{theorem}

We will generalize (1) and (2), and comment at the end on (3). In the polynomial case, we benefit from the control over the constant field extensions obtained in~\cite{AH:unicritical}. 

To start, we remind the reader of the relevant facts from algebraic number theory, primarily that decomposition subgroups are compatible with restriction. Recall that \(\bIMG = \Gal(\bar K(f^{-\infty}(t))/\bar K(t)\) is naturally isomorphic, by restriction, to \(\Gal(\wh K_f(f^{-\infty}(t))/\wh K_f(t)\).
\begin{proposition}\label{proposition: five lemma decomposition}
    Let \(a\in \bar K\) and assume \(a\) is not post-critical, and let \(L\) be a subfield of \(\bar K\) containing \(K(a)\). Let \(\mcO\) be the ring of integers of \(L(f^{-\infty}(t))\) and choose a maximal ideal \(\alpha\) of \(\mcO\) lying over \(t-a\). Let \(D=D_{\alpha/a}\) be the decomposition subgroup, the set of \(\sigma \in \Gal(L(f^{-\infty}(t))/L(t)) \subseteq \Gal(K(f^{-\infty}(t))/K(t)) = \IMG\) such that \(D^\sigma = D\). Let \(\bar D\) be its geometric counterpart, the set of \(\sigma \in D\) which fix \(\bar K\), or equivalently fix \(L\wh K_f\).
    
    Under the inclusion \(\bIMG \subseteq \IMG\), we have \(\bar D = D\cap \bIMG\). These groups fit into a commutative diagram, where the vertical diagrams are the natural inclusions and the rows are exact:
    \[\begin{tikzcd}
    	0 & {\bar D} & D & {D/\bar D} & 0 \\
    	0 & \bIMG & \IMG & {\IMG/\bIMG} & 0
    	\arrow[from=1-1, to=1-2]
    	\arrow[from=1-2, to=1-3]
    	\arrow[from=1-2, to=2-2]
    	\arrow[from=1-3, to=1-4]
    	\arrow[from=1-3, to=2-3]
    	\arrow[from=1-4, to=1-5]
    	\arrow[from=1-4, to=2-4]
    	\arrow[from=2-1, to=2-2]
    	\arrow[from=2-2, to=2-3]
    	\arrow[from=2-3, to=2-4]
    	\arrow[from=2-4, to=2-5]
    \end{tikzcd}\]

    Observe \(D/\bar D \cong \Gal(L\wh K_f / L)\) and \(\IMG/\bIMG \cong \Gal(\wh K_f/K)\); all the vertical arrows are induced by inclusions.
    
    Then \(D\to\IMG\) is surjective if and only if \(\bar D \to \bIMG\) and \(D/\bar D\to \IMG/\bIMG\) are surjective. The latter is surjective if and only if \(L\) and \(\wh K_f\) are linearly disjoint over \(K\): the index of \(D/\bar D\) in \(\IMG/\bIMG\) is the degree of \(\wh K_f \cap L\) over \(K\).
    
    In particular, if \(a\in K\) and \(L=K\), the specialization is surjective if and only if \(\bar D\to \bIMG\) is surjective.
\end{proposition}
\begin{proof}
    Recall the definition of a decomposition subgroup: let \(\mcO_\infty\) be the ring of integers of \(L(f^{-\infty}(t))\) and \(\alpha\) a maximal ideal of \(\mcO_\infty\) containing \(t-a\), then \(D_{\alpha/a}\) is the set of \(\sigma \in \Gal(L(f^{-\infty}(t)/L(t)) = \IMG\) such that \(D_{\alpha/a}^\sigma = D_{\alpha/a}\). These are the elements of the iterated monodromy group which can be reduced modulo \(\alpha\). The quotient \(\mcO_\infty/\alpha\) is a galois extension of \(L\) with galois group isomorphic to \(D\); the infinite case is reduced to the more usual finite case by identifying \(\alpha\) with the sequence of primes of \(L(f^{-n}(t))\) below it and passing to the limit.

    Now simply observe that \(L\) is algebraic over \(K\), so the constant field extension in \(L(f^{-\infty}(t))\) is the compositum of \(L\) and the constant field extension in \(K(f^{-\infty}(t))\). Therefore, \(\mcO_\infty\) contains the constant field extension \(L\wh K_f\) and so the quotient \(\mcO_\infty/\alpha\) does as well. Since \(L\wh K_f\) is a subfield of \(\mcO_\infty/\alpha\), galois theory ensures \(D_{\alpha/a}\) surjects onto \(\Gal(L\wh K_f/L)\). This coincides with the restriction from \(D_{\alpha/a}\) to \(\Gal(L\wh K_f/L)\) induced by the map to the quotient \(\Arb/\bArb\), so the kernel is precisely those \(\sigma \in D_{\alpha/a}\) that fix \(L\wh K_f\). The latter is the algebraic closure of \(L\) in \(L(f^{-\infty}(t))/L(t)\), so they are in the image of \(\bArb\).

    The surjectivity criterion is the four lemma.
\end{proof}

This allows us to use the Frattini subgroup of \(\bArb\), rather than \(\Arb\), to give a finite-level criterion for surjectivity.

\begin{theorem}\label{theorem: frattini specialization}
    Let \(f\) be a tamely ramified rational map over \(K\). Let \(a\in \bar K\) and \(L\) a subfield of \(\bar K\) containing \(K(a)\).
    
    If the Frattini subgroup of \(\bIMG\) is open, then there is an integer \(m\) such that the specialization of the profinite iterated monodromy group at \(t=a\) over \(L\) is surjective if and only if 
    \[\Gal(L\wh K_f(f^{-m}(a))/\wh K_f(a))\cong \Gal(\wh K_f(f^{-m}(t))/\wh K(t))\]
    and \(L\) is linearly disjoint from \(\wh K_f\) over \(K\). If \(a\in K\) and \(L=K\), the disjointness condition is vacuous.
\end{theorem}
\begin{proof}
    Every open subgroup of \(\bIMG\) contains a level stabilizer \(\bIMG_{(m)}\) for some sufficiently large \(m\). Therefore, any closed subgroup \(H\) of \(\bIMG\) such that \(H\bIMG_{(m)} = \bIMG\) also satisfies \(H\Phi = \bIMG\), which implies \(H = \bIMG\) because \(\Phi\) is the Frattini subgroup.

    Note that \(\Gal(L\wh K_f(f^{-m}(a))/ L\wh K_f)\) is precisely \(\bar D \bIMG_{(m)}/\bIMG_{(m)}\), and that it is naturally included in \(\bIMG/\bIMG_{(m)}\), which coincides with \(\Gal(\bar (f^{-m}(t))/\wh K_f(t))\). Both quotients are finite, so they are isomorphic if and only if \(\bar D \bIMG_{(m)} = \bIMG\), if and only if \(\bar D = \bIMG\) by Frattini. Since we have assumed linear disjointness of \(L\) and \(\wh K_f\) over \(K\), it follows from Proposition~\ref{proposition: five lemma decomposition} that this is equivalent to surjectivity of the specialization.
\end{proof}
\begin{corollary}\label{corollary: hilbertian constant field specialization}
    If \(\wh K_f\) is hilbertian, then there is a hilbert subset of \(a\in\wh K_f\) such that the specialization at \(t=a\) over \(\wh K_f\) is surjective.
\end{corollary}
\begin{proof}
    Taking \(L=\wh K_f\) in the theorem, surjectivity over \(\wh K_f\) only requires that
    \[\Gal(\wh K_f(f^{-m}(a))/\wh K_f(a)) \cong \Gal(\wh K_f(f^{-m}(t))/\wh K_f(t)),\]
    meaning a finite Galois extension of \(\wh K_f(t)\) is preserved by specialization, which is precisely what hilbertianity ensures.
\end{proof}
\begin{remark}
    
\end{remark}

Notice that when \(\wh K_f/K\) is finite, \(\wh K_f\) is hilbertian if \(K\) is hilbertian, in which case the hilbert subsets of \(\wh K_f\) contain hilbert subsets of \(K\), which allows one to reduce to hilbertianity of \(K\) and descend to \(K\). While \(\wh K_f/K\) is not finite in general, a milder group-theoretic property is still sufficient to carry out this argument.
\begin{definition}
    A galois extension \(L/K\) is said to be \emph{small} if \(\Gal(L/K)\) has only finitely many subgroups of index \(n\) for each integer \(n\).
\end{definition}

The following facts are proven in \cite[Chapter 19]{FJ:fieldarithmetic}.
\begin{proposition}\label{proposition: small implies descent}
    If \(L/K\) is small, then \(L\) is hilbertian and every hilbert subset of \(L\) contains a hilbert subset of \(K\). Two sufficient conditions for smallness are:
    \begin{enumalpha}
        \item If \(\Gal(L/K)\) is finitely generated, it is small.
        \item If \(K\) is a number field and \(L\) is ramified over only finitely many primes of \(K\), then \(\Gal(L/K)\) is small.
    \end{enumalpha}
\end{proposition}

In~\cite{BGJT:specializations}, the arithmetic pfIMG is pro-\(p\), so both the constant field Galois group and the geometric pfIMG are also pro-\(p\). For pro-\(p\) groups, smallness, open Frattini, and being finitely generated are all equivalent. More generally, those three are equivalent for a pronilpotent group whose order is divisible by only finitely many primes. 

What amounts to smallness of the constant field over number fields appears in~\cite[Theorem 2.5]{BGJT:specializations}. The main ingredient is the following, which follows from the Chevalley-Weil theorem. We reproduce an abbreviated version of this argument below, for comparison with our argument for polynomials over general hibertian fields:
\begin{lemma}[{\cite[Lemma 2.3]{BGJT:specializations}}]\label{lemma: tucker fix}
    Let \(f\) be a post-critically finite rational function over a number field \(K\). The constant field extension \(\wh K_f\) is ramified over finitely many primes.
\end{lemma}
\begin{proof}
    We can always find an \(a \in K\) which is not in the critical orbit, in which case \(K(f^{-\infty}(a))\) contains \(\wh K_f\) by Proposition~\ref{proposition: five lemma decomposition}. By~\cite[Theorem 8]{BIJJLMRSS:pcf}, the extension \(K(f^{-\infty}(a))\) is ramified over finitely many primes for any choice of \(a\), so \(\wh K_f\) is only ramified over a finite set of primes. Proposition~\ref{proposition: small implies descent}(b) furnishes smallness.
\end{proof}
\begin{proposition}
    If \(f\) is a PCF rational function over a number field \(K\), then \(\wh K_f\) is small over \(K\).
\end{proposition}
\begin{proof}
    Proposition~\ref{proposition: small implies descent}(b) is satisfied by Lemma~\ref{lemma: tucker fix}~.
\end{proof}

For polynomials, even post-critically infinite, we can show that \(\wh K_f\) for any hilbertian field by group-theoretic means.
\begin{proposition}\label{proposition: polynomial has small constant field}
    If \(f\) is a polynomial, then \(\wh K_f\) is small over \(K\).
\end{proposition}
\begin{proof}
    The branch cycle argument~\cite[Proposition 6.1]{AH:unicritical} shows that \(\wh K_f\) is contained in the pro-\(d\) cyclotomic extension of \(K\). Since \(\Gal(K(\zeta_{d^\infty})/K\) can be identified with a closed subgroup of \(\ZZ_d^*\), it is either finite or open, and in either case finitely generated. Therefore, its quotient \(\Gal(\wh K_f/K)\) is finitely generated, and therefore small by Proposition~\ref{proposition: small implies descent}(a).
\end{proof}

We now carry out the descent to \(K\).

\begin{corollary}
    Let \(f\) be a tamely ramified rational map over \(K\) such that \(\bArb\) has an open Frattini subgroup. If \(K\) is hilbertian and \(\Gal(\wh K_f/K)\) is small, then there are infinitely many \(a\in K\) such that the specialization at \(t=a\) over \(K\) is surjective. 

    Smallness obtains if either
    \begin{enumalpha}
        \item \(f\) is a polynomial and \(K\) is hilbertian, or
        \item \(f\) is a PCF rational function and \(K\) is a number field.
    \end{enumalpha}
\end{corollary}
\begin{proof}
    If \(K\) is hilbertian and \(\wh K_f\) is small over \(K\), then Proposition~\ref{proposition: small implies descent} ensures that \(\wh K_f\) is hilbertian and that its hilbert subsets also contain hilbert subsets of \(K\). The set \(\wh H\) of \(a\in \wh K_f\) of such that the specialization at \(t=a\) is surjective over \(\wh K_f\) is such a hilbert set and hence contains a hilbert subset \(H\) of \(K\). The critical orbit is a thin set, so \(H\) is still a hilbert subset after removing any points in the critical orbit. Therefore, there are infinitely many \(a\in K\) such that \(a\) is not a branch point and the specialization at \(t=a\) is surjective over \(\wh K_f\), and hence the specialization at \(t=a\) over \(K\) is surjective by \ref{theorem: frattini specialization}
\end{proof}

One way of framing the situation is as follows: based on~\cite{BGJT:specializations}, it is interesting to wonder whether Odoni's conjecture can be reduced to an open frattini subgroup for some large families of maps. Since the Frattini subgroup is pronilpotent, Theorem~\ref{theorem: self-similar upgrades} tells us that when the map is not an induced quotient, this requires a pronilpotent geometric profinite iterated monodromy. This leads us to expect that the argument will not work in as many cases, since pronilpotence is a strong limitation -- for polynomials, we saw in Corollary~\ref{corollary: open frattini implies pro-p} that this often implies pro-\(p\). In the induced quotient case, one expects for branch cycle reasons that the pfIMG will not typically be pronilpotent, even if they have an open Frattini subgroup, as is the case for twisted Chebyshev maps. 

While one might hope that there are other ways to reduce cases of Odoni's conjecture to hilberianity in non-pronilpotent situations, without using Frattini arguments, this turns out to be impossible in a precise sense. For example, when \(\bArb\) is the full automorphism group of the tree Dittmann and Kadets~\cite{DK:odoni} have shown that there are hilbertian fields over which no specialization is surjective. The full automorphism group of the tree is not pronilpotent, and in forthcoming work with Gotshall and Tucker, we show that non-pronilpotence alone is often sufficient to construct hilbertian fields for which surjectivity fails simultaneously for many families of rational maps.

\printbibliography
\end{document}

%% file: preamble.tex
\usepackage{amsmath}
\usepackage{amsthm}
\usepackage{amssymb}
\usepackage{amsfonts}
\usepackage[backend=biber,style=alphabetic,maxalphanames=99,maxbibnames=99]{biblatex}
\usepackage{booktabs}
\usepackage{enumitem}
\usepackage{framed}
\usepackage[margin=1in]{geometry}
\usepackage{hyperref}
\usepackage{multicol}
\usepackage{multirow}
\usepackage{quiver}
\usepackage{tcolorbox}
\usepackage{tikz-cd}
\usepackage{xcolor}

\tikzcdset{row sep/normal={5em,between origins}}
\tikzcdset{column sep/normal={5em,between origins}}

\definecolor{opheliacolor}{RGB}{82,151,248}
\colorlet{shadecolor}{opheliacolor!15}

\let\oldfootnote\footnote
\renewcommand{\footnote}[1]{\oldfootnote{\textcolor{opheliacolor}{#1}}}

\newlist{notations}{enumerate}{1}
\setlist[notations,1]{label={(N.\arabic*)},leftmargin=4em}
\newlist{assumptions}{enumerate}{1}
\setlist[assumptions,1]{label={(A.\arabic*)},leftmargin=4em}

\newlist{thmcases}{enumerate}{2}
\setlist[thmcases,1]{label={\textbf{Case \arabic*:}},leftmargin=4em}
\setlist[thmcases,2]{label={(\alph*)},leftmargin=3em}

\newlist{enumalpha}{enumerate}{1}
\setlist[enumalpha,1]{label={(\alph*)},leftmargin=3em}

\newlist{itemizeblank}{enumerate}{1}
\setlist[itemizeblank,1]{label={},leftmargin=3em}

\setlist[enumerate,1]{label={(\arabic*)},leftmargin=4em}

\hypersetup{
    colorlinks=true,
    linkcolor=blue,
    filecolor=magenta,
    urlcolor=blue,
}

\theoremstyle{definition}
\newtheorem{theorem}{Theorem}
\newtheorem{corollary}[theorem]{Corollary}
\newtheorem*{corollary*}{Corollary}
\newtheorem{lemma}[theorem]{Lemma}
\newtheorem{proposition}[theorem]{Proposition}
\newtheorem*{theorem*}{Theorem}

\theoremstyle{definition}

\newtheorem{definition}[theorem]{Definition}
\newtheorem{example}[theorem]{Example}

\newtheorem*{observation}{Observation}

\newtheorem{remark}[theorem]{Remark}

\newcommand{\rrrname}{}
\theoremstyle{definition}
\newtheorem*{rrr}{\rrrname}
\newenvironment{rerefresult}[1]
  {\renewcommand{\rrrname}{#1}%
   \begin{rrr}}
  {\end{rrr}}

\DeclareMathOperator{\Arb}{Arb}
\DeclareMathOperator{\bArb}{\overline{Arb}}
\DeclareMathOperator{\Aut}{Aut}

\DeclareMathOperator{\cchar}{char}
\DeclareMathOperator{\Gal}{Gal}
\DeclareMathOperator{\GG}{\mathbb G}
\DeclareMathOperator{\mcI}{\mathcal I}
\DeclareMathOperator{\id}{id}

\DeclareMathOperator{\IMG}{IMG}
\DeclareMathOperator{\bIMG}{\overline{IMG}}

\DeclareMathOperator{\mcO}{\mathcal O}
\DeclareMathOperator{\mcP}{\mathcal P}

\DeclareMathOperator{\PGL}{PGL}
\DeclareMathOperator{\PP}{\mathbb P}

\DeclareMathOperator{\mcQ}{\mathcal Q}
\DeclareMathOperator{\rad}{rad}

\DeclareMathOperator{\ZZ}{\mathbb Z}

\renewcommand{\epsilon}{\varepsilon}
\newcommand{\inv}{^{-1}}
\newcommand{\normal}{\triangleleft}
\newcommand{\odometer}{\omega}
\newcommand{\Odometer}{\Omega}%

\newcommand{\wh}{\widehat}

\makeatletter
\DeclareFontFamily{OMX}{MnSymbolE}{}
\DeclareSymbolFont{MnLargeSymbols}{OMX}{MnSymbolE}{m}{n}
\SetSymbolFont{MnLargeSymbols}{bold}{OMX}{MnSymbolE}{b}{n}
\DeclareFontShape{OMX}{MnSymbolE}{m}{n}{
    <-6>  MnSymbolE5
   <6-7>  MnSymbolE6
   <7-8>  MnSymbolE7
   <8-9>  MnSymbolE8
   <9-10> MnSymbolE9
  <10-12> MnSymbolE10
  <12->   MnSymbolE12
}{}
\DeclareFontShape{OMX}{MnSymbolE}{b}{n}{
    <-6>  MnSymbolE-Bold5
   <6-7>  MnSymbolE-Bold6
   <7-8>  MnSymbolE-Bold7
   <8-9>  MnSymbolE-Bold8
   <9-10> MnSymbolE-Bold9
  <10-12> MnSymbolE-Bold10
  <12->   MnSymbolE-Bold12
}{}

\let\llangle\@undefined
\let\rrangle\@undefined
\DeclareMathDelimiter{\llangle}{\mathopen}%
                     {MnLargeSymbols}{'164}{MnLargeSymbols}{'164}
\DeclareMathDelimiter{\rrangle}{\mathclose}%
                     {MnLargeSymbols}{'171}{MnLargeSymbols}{'171}
\makeatother